\documentclass[12pt, oneside, psamsfonts]{amsart}

\newif\ifPDF
\ifx\pdfoutput\undefined\PDFfalse
\else \ifnum \pdfoutput > 0 \PDFtrue
        \else \PDFfalse
        \fi
\fi

\usepackage[centertags]{amsmath}
\usepackage{amsfonts}
\usepackage{mathrsfs}
\usepackage{textcomp}
\usepackage{amssymb}
\usepackage{amsthm}
\usepackage{newlfont}
\usepackage[all]{xy}

\ifPDF
  \usepackage[pdftex]{color, graphicx}
  \usepackage[pdftex, bookmarks, colorlinks]{hyperref}
  \hypersetup{colorlinks=false}

\else
  \usepackage{color}
  \usepackage[dvips]{graphicx}
  \usepackage[dvips]{hyperref}
\fi

\usepackage{fixltx2e}
\usepackage{tkz-graph}
\tikzset{EdgeStyle/.style = {->}}
\tikzset{LabelStyle/.style= {fill=yellow}}
\usetikzlibrary{shapes,snakes,calendar,matrix,backgrounds,folding}

\usepackage[scale=0.8]{geometry}

\theoremstyle{plain}
\newtheorem{thm}{Theorem}[section]
\newtheorem{cor}[thm]{Corollary}

\newtheorem{lem}[thm]{Lemma}

\theoremstyle{definition}
\newtheorem{rem}[thm]{Remark}

\begin{document}

\title{On generators of $\mathrm{C}_r^*(\mathbb{F}_2)$}

\author{George A. Elliott}
\address{Department of Mathematics, University of Toronto, Toronto, Ontario, Canada~\ M5S 2E4}
\email{elliott@math.toronto.edu}

\author{Chun Guang Li}
\address{
School of Mathematics and Statistics, Northeast Normal University, Changchun, 130024, P. R.~China
}
\email{licg864@nenu.edu.cn}

\author{Zhuang Niu}
\address{Department of Mathematics, University of Wyoming, Laramie, WY, 82071, USA}
\email{zniu@uwyo.edu}

\begin{abstract}
It is shown that the reduced group C*-algebra $\mathrm{C}_r^*(\mathbb{F}_2)$ is singly generated, where $\mathbb F_2$ is the free group on two generators. With respect to the generator constructed in the paper, the topological free entropy dimension of  $\mathrm{C}_r^*(\mathbb{F}_2)$ is shown to be $2$.
\end{abstract}


\thanks{G.A.E.~is supported by an NSERC Discovery Grant. C.G.L.~is supported by an NNSF Grant of China (No. 12671150). Z.N.~is supported by a Simons Foundation Grant (MP-TSM-00002606).}

\date{\today}
\maketitle

\section{Introduction}
In \cite{Kadison-list}, among other questions, Kadison asked whether every von Neumann algebra acting on a separable Hilbert space is singly generated (see \cite{Ge_2003} for a survey of Kadison's problem list). Many von Neumann algebras are known to be singly generated, and the problem has been reduced to the $\mathrm{II}_1$ factor case (\cite{Willig_1974}), but it is still open in this case. For C*-algebras, the same question has been asked: whether every simple separable C*-algebra is singly generated. (Without simplicity, C*-algebras such as $\mathrm{C}(S^2)$ provide examples which are not singly generated; see \cite{Nagisa-04}.) Stable C*-algebras, UHF-absorbing C*-algebras, $\mathcal Z$-absorbing C*-algebras, and Villadsen's algebras of the first type, are known to be singly generated (\cite{Topping_1968}, \cite{Olsen_1976}, \cite{Li-Sh-AD}, \cite{Thiel-Winter-14}, and \cite{LNR-26}), while examples are known of simple unital AH algebras which are not singly generated (\cite{ELN}).

In this note, it is shown that the reduced group C*-algebra $\mathrm{C}_r^*(\mathbb F_2)$ is singly generated. In fact, there is a very simple generator (see \eqref{generator}).

\begin{thm}
With $g$ as given by \eqref{generator}, one has $\mathrm{C}^*_r(\mathbb F_2) = \mathrm{C}^*(g)$.
\end{thm}

With respect to the specified generator $g$, the topological free dimension of  $\mathrm{C}^*_r(\mathbb F_2)$ is $2$.

\begin{thm}
With respect to $g$ given by \eqref{generator}, one has
$$J + \frac{3}{2} + \log(2\pi) + \log{\frac{3}{5}} \leq \chi_{\mathrm{top}}(g) \leq J + \frac{3}{2} + \log(2\pi),$$
where
$$J = \frac{1}{2\pi} \int^{2\pi}_0 \log(\frac{3 + \cos\theta + \sqrt{(3+ \cos\theta)^2 - 1}}{2}) d\theta,$$
and 
$$\delta_{\mathrm{top}}(g) = 2,$$
where $\chi_{\mathrm{top}}$ and $\delta_{\mathrm{top}}$ are topological free entropy and topological free entropy dimension, respectively.
\end{thm}
\subsubsection*{Acknowledgements} The work was carried out with ChatGPT. The paper was reorganized, and the arguments were checked and reformulated, by the authors.

\section{The free shift model of $\mathrm{C}^*_r(\mathbb F_2)$ and conditional expectations}

Let us recall some basic facts concerning $\mathrm{C}^*_r(\mathbb F_2)$ and (reduced) amalgamated products.

\subsection{Free shift model}

Denote by $$\lambda: \mathbb F_2 \to \mathbb{B}(l^2(\mathbb F_2))$$ the left regular representation of $\mathbb F_2$. Consider reduced group C*-algebra $\mathrm{C}^*_r(\mathbb F_2)$, the C*-algebra generated by $\lambda(\mathbb F_2)$.

Write $\mathbb F_2 = \left< u, v\right>$. In the paper, we shall view $\mathbb F_2$ as the semidirect product
$$\mathbb F_2 = \mathbb F_\infty \rtimes \mathbb Z,$$
where $$\mathbb F_\infty = \left< w_n:= u^nvu^{-n},\ n \in \mathbb Z\right>$$  (see Example 5.4 in \cite{Smith}).
Then, we have the canonical decomposition
$$
\mathrm{C}^*_r(\mathbb F_2) \cong B \rtimes_\beta \mathbb Z,
$$
where
\begin{equation}\label{defn-B}
B := \mathrm{C}^*\{w_n; n \in \mathbb Z\} \subseteq \mathrm{C}^*_r(\mathbb F_2) \quad  \mathrm{and} \quad \beta(w_n) = w_{n+1}.
\end{equation}

Write
$$ (B, \tau) = \mathrm{C}^*_r(\mathbb F_\infty) = \ast_{n \in \mathbb Z}(\mathrm{C^*}(w_n), \tau),$$
where $\tau$ is the unique normalized trace (\cite{Powers})
and denote by
$$\mathbb{E}^{B}_n: B = \mathrm{C}^*_r(\mathbb F_\infty) \to \mathrm{C^*}(w_n)$$
the canonical conditional expectation. Note that
\begin{equation}\label{E0E1}
\mathbb{E}^{B}_{n}\circ \beta=\beta\circ \mathbb{E}^{B}_{n-1}, \quad n\in \mathbb{Z}.
\end{equation}

\subsection{Powers's property}
The following property of the conditional expectation (inspired by \cite{Powers}) seems to be well known to experts. Since the authors cannot find a reference, a proof is provided in the appendix.
\begin{thm}\label{CE}
Let
$$
 (A,\tau)=(M,\tau_M)*(N,\tau_N)
$$
be a reduced free product of unital tracial C*-algebras with faithful GNS
representations.  Assume that $M$ is abelian, and let $d\in M$ be
a Haar unitary (i.e., $\tau_M(d^k)=0$, $k\in\mathbb Z\setminus\{0\}$).
Then,
$$
 \mathbb{E}_M(a)=\lim_{L\to\infty}\frac1L\sum_{k=1}^L d^kad^{-k}, \quad a \in A,
$$
where $\mathbb E_M: A \to M$ is the canonical conditional expectation.
\end{thm}

Applying Theorem \ref{CE} to the free product decomposition $$(B, \tau) = (\mathrm{C}^*(w_0), \tau) \ast(\ast_{n\in\mathbb Z\setminus\{ 0\}} (\mathrm{C}^*(w_n), \tau)),$$  one has the following corollary.

\begin{cor}\label{average-in}
Let $D \subseteq B $ be a sub-C*-algebra, and suppose that there exists a Haar unitary $d \in D \cap \mathrm{C}^*(w_0)$. Then
$$\mathbb E^{B}_0(D) \subseteq D.$$
\end{cor}

\section{$\mathrm{C}^*_r(\mathbb F_2)$ is singly generated}

Put
\begin{equation}\label{per-ger}
 q=\exp\!\big(\frac{i}{10}(w_{-1}w_0+w_0^*w_{-1}^*)\big),
 \quad
 \mathrm{and}
 \quad
 h=3+w_0+w_0^*,
\end{equation}
and consider the element
\begin{equation}\label{generator}
 g=uqh,
\end{equation}
or, explicitly as a function of $u$ and $v$,
$$
g = u \exp(\frac{i}{10}(u^{-1}vuv + v^{-1}u^{-1}v^{-1}u))(3+v+v^{-1}).
$$
Then $g$ generates $\mathrm{C}^*_r(\mathbb F_2)$:
\begin{thm}\label{main-thm}
With $g$ as above, one has $$\mathrm{C}^*_r(\mathbb F_2) = \mathrm{C}^*(g). $$
\end{thm}

The proof of this theorem is based on the next four lemmas.

Note first that the element $uq$ is unitary. Also, 
\begin{equation}\label{h-in}
 h = |g|=(g^*g)^{1/2}\in \mathrm{C}^*(g).
\end{equation}
Since $\|v + v^*\| \leq 2$, the element $h = 3+v+v^*$ is invertible.

Since the unitary $q$ is in $B$, composing $\mathrm{Ad}(q)$ with $\beta$ (see \eqref{defn-B}), one obtains the automorphism 
\begin{equation}\label{gamma-in-B}
 \gamma:=\beta\circ\mathrm{Ad}(q) \in \mathrm{Aut}(B).
\end{equation}
Since $h \in B$ (see \eqref{defn-B}), the C*-algebra
\begin{equation}\label{defn-D}
 D := \mathrm{C}^*\{\gamma^n(h):\ n\in\mathbb{Z}\}
\end{equation}
is a $\gamma$-invariant sub-C*-algebra  of $B$.

Noting that $\gamma$ is implemented by the unitary
$$
 uq = gh^{-1}\in \mathrm{C}^*(g),
$$
one obtains 
\begin{equation*}\label{D-in-g}
\gamma^n(h) \in \mathrm{C}^*(g),\quad n \in \mathbb Z.
\end{equation*}
Thus, 
\begin{equation}\label{D-in-g}
 D \subseteq B \cap \mathrm{C}^*(g).
 \end{equation}

\begin{lem}\label{DinD}
With $D$ as in \eqref{defn-D}, one has $$\mathbb E^{B}_0(D) \subseteq D.$$
\end{lem}

\begin{proof}
By Corollary \eqref{average-in}, it is enough to find a Haar unitary in $ D\cap \mathrm{C}^*(w_0)$.
Put
$$
 a=w_0+w_0^*=h-3 \in D\cap \mathrm{C}^*(w_0)
$$
and consider the unitary 
$$
 d=\exp(2 i \arccos(a/2)\big)\in D\cap \mathrm{C}^*(w_0).
$$
We assert that $d$ is a Haar unitary.

Indeed, identify $\mathrm{C}^*(w_0)$ with $\mathrm{C}(\mathbb T)$ by $w_0=e^{i\theta}$ where $\theta \in [0, 2\pi]$. Then
\[
 \arccos(\cos\theta)=
 \begin{cases}
  \theta,&0\leq\theta\leq\pi,\\
  2\pi-\theta,&\pi\leq\theta\leq2\pi,
 \end{cases}
\]
and therefore
\[
 \tau(d^m)
 =\frac1{2\pi}
   \left(\int_0^\pi e^{2im\theta}\,d\theta
   +\int_\pi^{2\pi}e^{-2im\theta}\,d\theta\right)=0, \quad m \in \mathbb{Z} \setminus\{0\},
\]
as asserted.
\end{proof}

 \begin{lem}\label{recover-w-0}
With $D$ as in \eqref{defn-D}, one has
$$
w_0\in D.
$$
\end{lem}

\begin{proof}
Writing $$ z: = w_{-1}w_0, $$ one has
\begin{equation}\label{JA-expansion}
q=\exp(\frac{i}{10}(z+z^*))=\sum_{n\in \mathbb{Z}}c_nz^n,
\end{equation}
where
$$c_n=i^nJ_n({1}/{5}), $$
$$
J_n(x)=\sum_{k=0}^\infty\frac{(-1)^k}{k!(n+k)!}(\frac{x}{2})^{n+2k}, \quad n\geq0,
$$
and
$$
J_{n}(x)=(-1)^{-n}J_{-n}(x), \quad n<0.
$$
(see, for instance, the Jacobi-Anger expansion on pages 22--23 of \cite{Watson}).

The series \eqref{JA-expansion} converges absolutely in norm. It follows that
\begin{eqnarray*}
q(w_0+w_0^*)q^* & = & \sum_{n,m\in \mathbb{Z}}c_n\overline{c_m}(z^nw_0z^{-m}+z^nw_0^*z^{-m}) \\
& = & \sum_{n,m\in \mathbb{Z}}c_n\overline{c_m}((w_{-1}w_0)^n w_0(w_{-1}w_0)^{-m}+(w_{-1}w_0)^nw_0^*(w_{-1}w_0)^{-m}).
\end{eqnarray*}
Since $w_{-1}$ and $w_0$ are freely independent, we have
$$
\mathbb{E}^B_{-1}((w_{-1}w_0)^nw_0(w_{-1}w_0)^{-m})=\begin{cases}
w_{-1}^*, & n=0,\  m=1 \\
0, & \mathrm{otherwise}
\end{cases},
$$
and
$$
\mathbb{E}^B_{-1}((w_{-1}w_0)^nw_0^*(w_{-1}w_0)^{-m})=\begin{cases}
w_{-1}, & n=1,\ m=0 \\
0, & \mathrm{otherwise}
\end{cases}.
$$
It then follows that
\begin{eqnarray*}
\mathbb{E}^B_{-1}(q(w_0+w_0^*)q^*)&= &c_1\overline{c_0}w_{-1}+c_0\overline{c_1}w_{-1}^*\\
&= &iJ_1({1}/{5})\overline{J_0({1}/{5})}w_{-1}+J_0({1}/{5})\overline{iJ_1({1}/{5})}w_{-1}^*\\
&=& i\kappa(w_{-1}-w_{-1}^*),
\end{eqnarray*}
where
$$\kappa=J_0({1}/{5})J_1({1}/{5})>0.$$
Together with  \eqref{E0E1}, one has 
\begin{eqnarray*}
\mathbb{E}^B_0(\gamma(h))&= & \mathbb{E}^B_0(\beta(qhq^*))\\
&= &\beta(\mathbb{E}^B_{-1}(q(3+w_0+w_0^*)q^*))\\
&=& 3+i\kappa\beta(w_{-1}-w_{-1}^*)\\
&=& 3+i\kappa(w_{0}-w_{0}^*).
\end{eqnarray*}
Note that $\gamma(h)\in D$ (see \eqref{defn-D}). It follows from  Lemma \ref{DinD} that
$$
w_0-w_0^* = (\mathbb{E}_0^{B}(\gamma(h)) - 3)/i\kappa\in D.
$$
Since $w_0+w_0^*=h-3\in D$, one has
$$
w_0=\frac{1}{2}(w_0+w_0^*)+\frac{1}{2}(w_0-w_0^*)\in D,
$$
as asserted.
\end{proof}

\begin{lem}\label{C:Duhamel}
For any self-adjoint elements $a, b$ in a C*-algebra, one has
$$
\|e^{ia}-e^{ib}\|\leq \|a-b\|.
$$
\end{lem}
\begin{proof}
By Duhamel's formula (see, for instance, \cite{Engel}, pages 161, 162), 
$$
e^{ia}-e^{ib}=i\int_0^1e^{i(1-s)a}(a-b)e^{isb}ds.
$$
The lemma follows.
\end{proof}

 \begin{lem}\label{recover-q}
With $D$ as in \eqref{defn-D}, one has
$$
q\in D.
$$
\end{lem}

\begin{proof}
By Lemma \ref{recover-w-0}, $w_0\in D$. Since $D$ is $\gamma$-invariant, one has (see \eqref{gamma-in-B})
\begin{equation}\label{qwq-in-D}
q^*w_{-1}q=q^*\beta^{-1}(w_0)q=\gamma^{-1}(w_0)\in D.
\end{equation}
Define the maps
$$
Q: B \ni b \mapsto \exp(\frac{i}{10}(bw_{0}+w_{0}^*b^*)) \in B$$
and
$$\Phi: B \ni b \mapsto Q(b)q^*w_{-1}qQ(b)^* \in B.$$
(Recall that by definition $w_n$ and $q$ are in $B$.)

Then, for $b_1, b_2\in B$, by Lemma \ref{C:Duhamel}, one has
\begin{eqnarray*}
\|\Phi(b_1)-\Phi(b_2)\|&= &\|Q(b_1)q^*w_{-1}qQ(b_1)^*-Q(b_2)q^*w_{-1}qQ(b_2)^*\|\\
&\leq & \|Q(b_1)q^*w_{-1}qQ(b_1)^*-Q(b_2)q^*w_{-1}qQ(b_1)^*\|\\
&&+\|Q(b_2)q^*w_{-1}qQ(b_1)^*-Q(b_2)q^*w_{-1}qQ(b_2)^*\|\\
&\leq & 2\|Q(b_1)-Q(b_2)\|\\
&\leq & \frac{1}{5}\|(b_1w_0+w_0^*b_1^*)-(b_2w_0+w_0^*b_2^*)\|\\
&\leq & \frac{2}{5}\|b_1-b_2\|.
\end{eqnarray*}
Note that, by \eqref{per-ger} (the definition of $q$), 
$$Q(w_{-1}) = \exp(\frac{i}{10}(w_{-1}w_0 + w_0^*w_{-1}^*)) = q, $$
and so
$$\Phi(w_{-1})= q q^*w_{-1}q q^* = w_{-1}.$$
Then, for each $n\in \mathbb{N}$, iterating the inequality above, one has
\begin{equation*}
\|\Phi^n(0)- w_{-1}\| =  \|\Phi^n(0)-\Phi^n(w_{-1})\|  \leq  (\frac{2}{5})^{n}\| 0-w_{-1}\|.
\end{equation*}
That is,
\begin{equation}\label{w-1-lim}
 w_{-1} = \lim_{n\to \infty} \Phi^n(0).
 \end{equation}

A straightforward calculation shows that
$$\Phi(w_0),\ \Phi(q^*w_{-1}q) \in \mathrm{C}^*(w_0, q^*w_{-1}q)$$
and
$$\Phi(0) \in  \mathrm{C}^*(w_0, q^*w_{-1}q).$$
Therefore,
$$\Phi^n(0) \in \mathrm{C}^*(w_0, q^*w_{-1}q), \quad n\in \mathbb N.$$
Hence, by \eqref{w-1-lim},
$$w_{-1} \in \mathrm{C}^*(w_0, q^*w_{-1}q). $$
Since $w_0 \in D$ (by Lemma \ref{recover-w-0}) and $q^*w_{-1}q \in D$ (see \eqref{qwq-in-D}), one has
$$\mathrm{C}^*(w_0, q^*w_{-1}q) \subseteq  D, $$
and therefore
$$
q=\exp(\frac{i}{10}(w_{-1}w_0+w_0^*w_{-1}^*))\in D,
$$
as desired.
\end{proof}

Let us prove the main theorem.
\begin{proof}[Proof of Theorem \ref{main-thm}]
By \eqref{h-in}, Lemma \ref{recover-q}, and \eqref{D-in-g},  respectively, one has $$h, q \in \mathrm{C}^*(g),$$
and therefore
$$u = uqhh^{-1}q^{-1} = gh^{-1}q^{-1} \in \mathrm{C}^*(g). $$
By Lemma \ref{recover-w-0} and \eqref{D-in-g}, $$v = w_0 \in D \subseteq \mathrm{C}^*(g).$$
Thus,
$$\mathrm{C}^*_r(\mathbb F_2) = \mathrm{C}^*(u, v) \subseteq \mathrm{C}^*(g) \subseteq \mathrm{C}^*_r(\mathbb F_2),$$
as stated.
\end{proof}

\begin{rem}
The formula \eqref{generator} was provided by ChatGPT. Is there some fundamental mechanism behind this formula?
\end{rem}

\section{Free entropy and free entropy dimension of the generator}\label{sec:free-entropy}

As a further revelation of the usefulness of ChatGPT, let us determine, or rather estimate the topological free entropy and topological free entropy dimension of $\mathrm{C}_r^*(g)$, with respect to the generator $g$ given by \eqref{generator}.

Let $\tau$ denote the canonical trace on $\mathrm{C}^*_r(\mathbb F_2)$, and retain the
notation
\[
 w_{-1}=u^{-1}vu,\qquad
 q=\exp\!\left(\frac{i}{10}
       (w_{-1}v+v^*w_{-1}^*)\right),\qquad
 h=3+v+v^*,\qquad
 g=uqh.
\]
Thus, $g^*g=h^2$ and $1\leq h\leq 5$. The microstates free entropy below is
computed in the tracial $W^*$-probability space $(L(\mathbb F_2),\tau)$.
Since microstates free entropy and topological free entropy are defined for self-adjoint
tuples, throughout this section we use the conventions
\[
 \chi(g)=\chi^{\mathrm{sa}}(\Re(g),\Im(g)),
 \qquad
 \chi_{\mathrm{top}}(g)
 =\chi_{\mathrm{top}}^{\mathrm{sa}}
       (\Re(g),\Im(g)),
\]
and similarly for the corresponding free entropy dimensions.  Equivalently, one
may use the non-self-adjoint microstate convention of
\cite{NicaShlyakhtenkoSpeicher}: its Euclidean measure on $\mathrm{M}_k(\mathbb C)$
and normalization term $\log k$ agree with the product Euclidean measure on
$(M_k^{\mathrm{sa}})^2$ and its normalization term $\log k$ under the map 
$T\mapsto(\Re T, \Im T)$.

Put
\begin{equation}
 c_{\mathrm U}:=\frac34+\frac12\log(2\pi)
 \label{eq:unitary-volume-constant}
\end{equation}
and
\begin{equation}
 \begin{split}
 \mathcal J
 &:=\frac{1}{(2\pi)^2}
     \int_0^{2\pi}\!\int_0^{2\pi}
       \log(3+\cos\theta+\cos\varphi)\,d\theta\,d\varphi \\
 &=\frac{1}{2\pi}\int_0^{2\pi}
   \log\!\left(
     \frac{3+\cos\theta+
       \sqrt{(3+\cos\theta)^2-1}}{2}
   \right)d\theta.
 \end{split}
 \label{eq:J-constant}
\end{equation}

\begin{thm}
\label{thm:entropy-of-g}
With Voiculescu's normalization of microstates free entropy and topological
free entropy, one has
\begin{equation}
 \mathcal J+2c_{\mathrm U}+\log\frac35
 \ \leq\ \chi_{\mathrm{top}}(g)
 \ \leq\ \chi(g)
 \ \leq\ \mathcal J+2c_{\mathrm U}.
 \label{eq:entropy-bounds}
\end{equation}
Consequently,
\begin{equation}
 \delta_{\mathrm{top}}(\Re g, \Im g)
 =\delta_0(\Re g, \Im g;\tau)
 =\delta(\Re g, \Im g;\tau)=2.
 \label{eq:entropy-dimensions-g}
\end{equation}
\end{thm}

The argument below gives bounds (those of \eqref{eq:entropy-bounds}) rather than the exact value of either entropy.
The gap in \eqref{eq:entropy-bounds} is $\log(5/3)$; it comes from a uniform
lower bound for the Jacobian of the twisting map $U\mapsto Uq(U,V)$.

Once $\chi_{\mathrm{top}}(g)$ is shown to be finite (\eqref{eq:entropy-bounds}), it follows from \cite{VoiculescuTopological} that the topological free dimension must be $2$.

\subsection{The upper bound of \eqref{eq:entropy-bounds}}

We first compare topological entropy and tracial microstates entropy.  Voiculescu's
free-capacity inequality \cite[Section~3]{VoiculescuTopological}, together
with uniqueness of the tracial state on $\mathrm{C}^*_r(\mathbb F_2)$ and the identification 
$\mathrm{C}^*(\Re g,\Im g)=\mathrm{C}^*_r(\mathbb F_2)$, gives
\begin{equation}
 \chi_{\mathrm{top}}(g)\leq \chi(g).
 \label{eq:top-less-than-tracial}
\end{equation}
Indeed, if sufficiently fine norm-microstate conditions did not force a
prescribed finite collection of $\tau$-moment conditions, a sequence of
counterexamples and a free ultrafilter would produce a tracial state on
$C_r^*(\mathbb F_2)$ different from $\tau$.

Let $r$ be a Haar unitary which is $*$-free from $h$.  The maximality theorem
for the microstates free entropy of $R$-diagonal elements
\cite[Theorem~2.1]{NicaShlyakhtenkoSpeicher} gives
\begin{equation}
 \chi(g)\leq \chi(rh)
 =\chi^{\mathrm{sa}}\!\left(\frac{h^2}{2}\right)+c_{\mathrm U}.
 \label{eq:R-diagonal-upper}
\end{equation}
We calculate the last one-variable entropy.  With respect to $\tau$, the
spectral distribution of $v$ is normalized Haar measure on the circle.  Set
\[
 H(\theta)=3+2\cos\theta,
 \qquad
 P(\theta)=\frac12 H(\theta)^2.
\]
The one-variable formula for free entropy yields
\begin{equation}
 \chi^{\mathrm{sa}}\!\left(\frac{h^2}{2}\right)
 =\frac{1}{(2\pi)^2}\int_0^{2\pi}\!\int_0^{2\pi}
       \log|P(\theta)-P(\varphi)|\,d\theta\,d\varphi
   +c_{\mathrm U}.
 \label{eq:one-variable-entropy-P}
\end{equation}
Now
\[
 P(\theta)-P(\varphi)
 =\bigl(H(\theta)-H(\varphi)\bigr)
   \frac{H(\theta)+H(\varphi)}{2}.
\]
Writing $\lambda=e^{i\theta}$ and $\mu=e^{i\varphi}$, we also have
\[
 H(\theta)-H(\varphi)
 = (\lambda-\mu)\left(1-\frac{1}{\lambda\mu}\right).
\]
The logarithmic potential of normalized Haar measure on the circle vanishes
on the circle.  Consequently,
\[
 \frac{1}{(2\pi)^2}\int_0^{2\pi}\!\int_0^{2\pi}
       \log|H(\theta)-H(\varphi)|\,d\theta\,d\varphi=0.
\]
The remaining term in \eqref{eq:one-variable-entropy-P} is precisely
$\mathcal J$.  Hence
\begin{equation}
 \chi^{\mathrm{sa}}\!\left(\frac{h^2}{2}\right)
 =\mathcal J+c_{\mathrm U}.
 \label{eq:entropy-modulus}
\end{equation}
Equations \eqref{eq:top-less-than-tracial},
\eqref{eq:R-diagonal-upper}, and \eqref{eq:entropy-modulus} prove the two
upper inequalities in \eqref{eq:entropy-bounds}.

\subsection{A Jacobian estimate}

Equip $M_k(\mathbb C)$ with the real Hilbert--Schmidt inner product
\[
 \langle A,B\rangle=\Re \operatorname{Tr}(AB^*)
\]
and use the induced Riemannian metric on $\mathrm U(k)$.  Let $\mu_k$ be
normalized Haar measure on $\mathrm U(k)$ and let $C_k$ be its Riemannian
volume.  Thus Riemannian volume is $C_k\mu_k$, and
\begin{equation}
 \lim_{k\to\infty}
 \left(\frac{1}{k^2}\log C_k+\frac12\log k\right)
 =c_{\mathrm U};
 \label{eq:unitary-volume-asymptotic}
\end{equation}
see, for example, \cite[Lemma~2.8]{NicaShlyakhtenkoSpeicher}.

For $U,V\in\mathrm U(k)$, define
\begin{equation}
 \begin{split}
 Z(U,V)&=U^*VUV,\\
 Q(U,V)&=\exp\!\left(\frac{i}{10}
                  (Z(U,V)+Z(U,V)^*)\right),\\
 H(V)&=3+V+V^*,\\
 \Phi_k(U,V)&=UQ(U,V)H(V).
 \end{split}
 \label{eq:matrix-model-map}
\end{equation}

\begin{lem}
\label{lem:twist-Jacobian}
For fixed $V\in\mathrm U(k)$, the map
\[
 \psi_V\colon \mathrm U(k)\longrightarrow\mathrm U(k),
 \qquad \psi_V(U)=UQ(U,V),
\]
is injective, and every singular value of its differential is at least
$3/5$.  Consequently, the real Jacobian of
\[
 \Theta_k(U,V)=(UQ(U,V),V)
\]
is at least $(3/5)^{k^2}$.
\end{lem}

\begin{proof}
For $U_1,U_2\in\mathrm U(k)$,
\[
 \|Z(U_1,V)-Z(U_2,V)\|_2
 \leq 2\|U_1-U_2\|_2.
\]
The Duhamel formula, valid for every unitarily invariant norm (see Lemma \ref{C:Duhamel} and proof), gives
$\|e^{iA}-e^{iB}\|_2\leq\|A-B\|_2$ for self-adjoint $A,B$.  It follows that
\[
 \|Q(U_1,V)-Q(U_2,V)\|_2
 \leq \frac25\|U_1-U_2\|_2.
\]
Therefore
\[
 \|\psi_V(U_1)-\psi_V(U_2)\|_2
 \geq \frac35\|U_1-U_2\|_2.
\]
This proves injectivity.  Differentiating the last inequality along smooth
curves shows that every singular value of $D\psi_V$ is at least $3/5$.
Since $\dim_{\mathbb R}\mathrm U(k)=k^2$ and $D\Theta_k$ is block triangular,
the asserted Jacobian bound follows.
\end{proof}

Set
\begin{equation}
 \mathcal P_k(V)=\frac12 H(V)^2.
 \label{eq:P-map}
\end{equation}

\begin{lem}
\label{lem:spectral-Jacobian}
Suppose that the eigenvalues of $V$ are $e^{i\theta_1},\ldots,e^{i\theta_k}$,
and put $h_j=3+2\cos\theta_j$.  At every regular point of $\mathcal P_k$,
its real Jacobian is
\begin{equation}
 \operatorname{Jac}\mathcal P_k(V)
 =\prod_{i,j=1}^k
   \frac{h_i+h_j}{2}
   \left|e^{i\theta_i}-e^{-i\theta_j}\right|.
 \label{eq:P-Jacobian}
\end{equation}
\end{lem}

\begin{proof}
By unitary equivariance, it is enough to calculate at a diagonal $V$.
On the diagonal tangent direction which changes $\theta_i$, the absolute
derivative of $h_i^2/2$ is
\[
 2h_i|\sin\theta_i|
 =h_i|e^{i\theta_i}-e^{-i\theta_i}|.
\]
For $i<j$, write the tangent vector at $V$ as $VK$, where $K^*=-K$, and put
$K_{ij}=z$.  The corresponding $(i,j)$ entry of $D H(V)$ is
$(e^{i\theta_i}-e^{-i\theta_j})z$, and hence that of
$D\mathcal P_k(V)$ is
\[
 \frac{h_i+h_j}{2}
 (e^{i\theta_i}-e^{-i\theta_j})z.
\]
The determinant on this real two-dimensional off-diagonal tangent plane is
the square of the modulus of the displayed scalar.  Multiplication over all
diagonal directions and all pairs $i<j$ gives \eqref{eq:P-Jacobian}.
\end{proof}

The polar-coordinate map
\begin{equation}
 \mathrm U(k)\times M_k^{\mathrm{sa},+}\longrightarrow M_k(\mathbb C),
 \qquad (W,P)\longmapsto W(2P)^{1/2},
 \label{eq:polar-coordinate-map}
\end{equation}
has Jacobian one when the first factor has its Hilbert--Schmidt Riemannian
volume and the second and target have Euclidean Lebesgue measure.  Equivalently,
when normalized Haar measure is used on the first factor, the Jacobian factor
is $C_k$; see the polar-decomposition calculation in
\cite[Section~2]{NicaShlyakhtenkoSpeicher}.  The factorization
\[
 (U,V)\xrightarrow{\ \Theta_k\ }(UQ(U,V),V)
 \longmapsto (UQ(U,V),\mathcal P_k(V))
 \xrightarrow{\ \eqref{eq:polar-coordinate-map}\ }\Phi_k(U,V)
\]
therefore combines Lemmas~\ref{lem:twist-Jacobian} and
\ref{lem:spectral-Jacobian} into a lower bound for
$\operatorname{Jac}\Phi_k$.

\subsection{Phase averaging and the lower bound of \eqref{eq:entropy-bounds}}

For $V$ as in Lemma~\ref{lem:spectral-Jacobian}, put
\begin{align}
 A_k(V)
 &:=\frac{1}{k^2}\sum_{i,j=1}^k
       \log\!\left(\frac{h_i+h_j}{2}\right),
 \label{eq:A-k}\\
 B_k(V)
 &:=\frac{1}{k^2}\sum_{i,j=1}^k
       \log|e^{i\theta_i}-e^{-i\theta_j}|.
 \label{eq:B-k}
\end{align}
Thus
\begin{equation}
 \frac{1}{k^2}\log\operatorname{Jac}\mathcal P_k(V)
 =A_k(V)+B_k(V).
 \label{eq:log-P-Jacobian}
\end{equation}

If the empirical eigenvalue measures of $V_k$ converge weakly to normalized
Haar measure on the circle, then
\begin{equation}
 \sup_{t\in\mathbb R}
 |A_k(e^{it}V_k)-\mathcal J|\longrightarrow 0.
 \label{eq:A-uniform-limit}
\end{equation}
This follows because the function
$(\theta,\varphi)\mapsto\log(3+\cos\theta+\cos\varphi)$ is continuous, and
the family of its simultaneous translates is compact in
$C(\mathbb T^2)$.

The singular term $B_k$ is controlled by averaging the scalar phase.  For
every $V$ outside a null set,
\begin{equation}
 \frac{1}{2\pi}\int_0^{2\pi}B_k(e^{it}V)\,dt=0,
 \qquad B_k(e^{it}V)\leq\log 2.
 \label{eq:B-phase-average}
\end{equation}
Indeed, each summand has the form
\[
 \log|e^{i(\theta_i+\theta_j+2t)}-1|,
\]
the integral of which is zero.  It follows from \eqref{eq:B-phase-average} that,
for every $\eta>0$,
\begin{equation}
 \frac{1}{2\pi}
 \left|\{t\in[0,2\pi]:B_k(e^{it}V)\geq-\eta\}\right|
 \geq \frac{\eta}{\log 2+\eta}.
 \label{eq:good-phase-proportion}
\end{equation}

We now pass to topological microstates.  Let $U_k,V_k$ be independent
Haar-distributed unitary matrices.  Strong asymptotic freeness
\cite{CollinsMale} implies that $(U_k,V_k)$ converges almost surely strongly
to the canonical free Haar pair $(u,v)$.  The circle of automorphisms
\[
 \alpha_t(u)=u,\qquad \alpha_t(v)=e^{it}v
\]
preserves both the norm and the canonical trace, and
\[
 \Phi_k(U,e^{it}V)
\]
is the matrix expression corresponding to $\alpha_t(g)$.  Strong convergence,
a finite-net argument in $t$, and the uniform continuity of the expressions
in \eqref{eq:matrix-model-map} therefore have the following consequence.  For
every fixed topological microstate neighborhood of $g$, there are
phase-invariant Borel sets
\[
 \Omega_k\subseteq\mathrm U(k)^2,
 \qquad (\mu_k\times\mu_k)(\Omega_k)\longrightarrow 1,
 \label{eq:Omega-k}
\]
such that
\begin{enumerate}
 \item $\Phi_k(U,V)$ belongs to that microstate neighborhood whenever
       $(U,V)\in\Omega_k$;
 \item $A_k(e^{it}V)\geq\mathcal J-\eta$ for every $t$ whenever
       $(U,V)\in\Omega_k$, after increasing $k$ if necessary.
\end{enumerate}
Here phase-invariance means that $(U,V)\in\Omega_k$ if and only if
$(U,e^{it}V)\in\Omega_k$ for every $t$.

Define
\[
 E_k=\{(U,V)\in\Omega_k:B_k(V)\geq-\eta\}.
\]
By \eqref{eq:good-phase-proportion} and phase-invariance,
\begin{equation}
 (\mu_k\times\mu_k)(E_k)
 \geq\frac{\eta}{\log 2+\eta}
       (\mu_k\times\mu_k)(\Omega_k).
 \label{eq:E-k-measure}
\end{equation}
Moreover, \eqref{eq:log-P-Jacobian} gives
\begin{equation}
 \operatorname{Jac}\mathcal P_k(V)
 \geq \exp\bigl(k^2(\mathcal J-2\eta)\bigr),
 \qquad (U,V)\in E_k.
 \label{eq:P-Jacobian-lower}
\end{equation}

It remains to control multiplicity.  The polar decomposition of
$\Phi_k(U,V)$ determines $W=UQ(U,V)$ and $P=\mathcal P_k(V)$.  For almost
every $P$, its spectrum is simple.  For such a $P$, the equation
\[
 3+V+V^*=(2P)^{1/2}
\]
has at most $2^k$ solutions.  Indeed, $V$ commutes with $V+V^*$; hence it
preserves each one-dimensional spectral subspace of $(2P)^{1/2}-3$, and on
each such subspace there are at most two choices for its eigenvalue.  For each
fixed $V$, Lemma~\ref{lem:twist-Jacobian} shows that $U$ is uniquely determined
by $W$.  Thus the multiplicity of $\Phi_k$ is at most $2^k$ almost everywhere.

Let $\lambda_k$ denote Euclidean Lebesgue measure on $M_k(\mathbb C)$.  The
area formula, the two Jacobian lemmas, \eqref{eq:P-Jacobian-lower}, and
\eqref{eq:E-k-measure} give
\begin{equation}
 \begin{split}
 \lambda_k(\Phi_k(E_k))
 &\geq 2^{-k}\left(\frac35\right)^{k^2}
       e^{k^2(\mathcal J-2\eta)} C_k^2
       (\mu_k\times\mu_k)(E_k).
 \end{split}
 \label{eq:image-volume-lower}
\end{equation}
Since $\Phi_k(E_k)$ is contained in the prescribed topological microstate
neighborhood of $g$, we take $k^{-2}\log$, add the normalization term
$\log k$, use \eqref{eq:unitary-volume-asymptotic}, and then let
$k$ tend to $\infty$.  The factors $2^{-k}$ and
$(\mu_k\times\mu_k)(E_k)$ do not contribute on the $k^2$ scale.  We obtain
\[
 \chi_{\mathrm{top}}(g)
 \geq \mathcal J-2\eta+2c_{\mathrm U}+\log\frac35.
\]
Letting $\eta\downarrow0$ proves the lower inequality in
\eqref{eq:entropy-bounds}.

Finally, finite topological free entropy for a two-variable self-adjoint tuple
forces its topological free entropy dimension to be $2$
(\cite{VoiculescuTopological}).  Likewise, finite microstates free entropy
forces both $\delta$ and $\delta_0$ to equal the number of self-adjoint
variables (\cite{VoiculescuEntropyII}).  This proves
\eqref{eq:entropy-dimensions-g}.

%

\appendix

\section{Proof of Theorem \ref{CE}}

We follow the argument of \cite{Choi}, proving that $\mathrm{C}^*_r(\mathbb Z_2 \ast \mathbb Z_3)$ is simple, has unique trace, and is not quasidiagonal  (compare also the early paper of Powers, \cite{Powers}). First, recall the following lemma:
\begin{lem}[Lemma 2.2 of \cite{Choi}]\label{abs-avg}
Let $\mathcal H = \mathcal M \oplus \mathcal M^\perp$ be a direct sum decomposition of a Hilbert space $\mathcal H$. Suppose that $x \in \mathbb B(\mathcal H)$ has the form $\left( \begin{array}{cc} 0 & \ast \\ \ast & \ast \end{array} \right)$ and $u_i \in \mathbb B(\mathcal H)$, $i=1, ..., n$, are unitaries such that $u_iu_j^*$ have the form $\left( \begin{array}{cc} \ast & \ast \\ \ast & 0 \end{array} \right)$ when $i \neq j$. Then 
$$ \| \sum_{i=1}^n u_ixu_i \| \leq 2 \sqrt{n} \| x\|.$$
\end{lem}

Let us also recall the construction of the reduced free product of tracial C*-algebras. Let $(M, \tau_M)$ and $(N, \tau_N)$ be unital tracial C*-algebras with faithful GNS
representations. 
A reduced word is a product $a_1\cdots a_n$, where each $a_j$ is contained in $\ker \tau_M$ or $\ker \tau_N$, and adjacent letters are contained in different factors.
The canonical conditional expectation $\mathbb{E}_M: A\rightarrow M$ is defined so that
$\mathbb{E}_M(a)=a$ for each $a\in M$ and $\mathbb{E}_M(b)=0$ for each reduced word containing
at least a letter in $\ker\tau_N$.

For $k\in\{M, N\}$, set
$$
H_k=L^2(k,\tau_k),\quad H_k'=H_k\ominus \mathbb{C}\widehat{1},
$$
where $\widehat{a}$ denotes the GNS vector of $a\in k$, and the inner product in $H_k$ is of course
$$
\langle \widehat{a}, \widehat{b}\rangle :=\tau_k(b^*a).
$$
The reduced free-product Hilbert space is
\begin{equation}\label{eq:Fock}
\mathcal{F}=\mathbb{C}\Omega\oplus\bigoplus_{\substack{
 n\in \mathbb{N},~~k_1,\ldots,k_n\in\{M,N\}\\
 k_j\neq k_{j+1}(1\leq j<n) }
}H'_{k_1}\otimes\cdots\otimes H_{k_n}',
\end{equation}
where $\Omega$ is a unit vector which is called the vacuum.

The (reduced) free-product C*-algebra 
$$
 (B,\tau)=(M,\tau_M)*(N,\tau_N)
$$
is the C*-algebra generated by
the left actions of $M$ and $N$ on $\mathcal{F}$. In fact,
if $a\in \ker\tau_k$, then
\begin{equation*}
a\Omega=\widehat{a},\quad a(\xi_1\otimes\cdots\otimes \xi_n)=\widehat{a}\otimes \xi_1\otimes\cdots\otimes \xi_n,\quad \text{when} \ \
\xi_1\in H_j',\  j\neq k.
\end{equation*}


\begin{lem}\label{L:norm}
Let
$$
 (A,\tau)=(M,\tau_M)*(N,\tau_N)
$$
be a reduced free product of unital tracial C*-algebras with faithful GNS
representations.
Let $d\in M$ be a Haar unitary.
If $x=a_1\cdots a_n$ is a reduced word with $a_1 \in \ker\tau_N$, then for each $L\in \mathbb{N}$, we have
$$
\|\sum_{k=1}^Ld^kxd^{-k}\|\leq 2\sqrt{L}\|x\|.
$$
\end{lem}

\begin{proof}
Consider 
$$ 
\mathcal M := \overline{ \mathrm{span}\{ \xi_{1} \otimes \cdots \otimes \xi_{m}: \xi_{1} \in H'_M,\ m\in\mathbb N \} }.
$$
Then
$$ 
\mathcal M^\perp = \overline{ \mathrm{span}\{\Omega,\ \xi_{1} \otimes \cdots \otimes \xi_{m}: \xi_{1} \in H_N',\ m \in \mathbb N\} }.
$$
Since $x = a_1\cdots a_n$ with $a_1 \in \ker\tau_N$, one has that $x$ has the form $\left( \begin{array}{cc} 0 & \ast \\ \ast & \ast \end{array} \right)$ with respect to $\mathcal F = \mathcal M \oplus \mathcal M^\perp $. Also note that, for each $i\neq j$,  the element $d^i(d^j)^* = d^{i-j} \in \ker\tau_M$ has the form $\left( \begin{array}{cc} \ast & \ast \\ \ast & 0 \end{array} \right)$ with respect to $\mathcal F = \mathcal M \oplus \mathcal M^\perp $. Thus, the desired statement follows from Lemma \ref{abs-avg} directly.
\end{proof}

\begin{proof}[Proof of Theorem \ref{CE}]
For $L\in \mathbb{N}$, define
$$
A_L(a)=\frac{1}{L}\sum_{k=1}^Ld^kad^{-k},\quad a\in A.
$$
For any $b$ in the algebraic $*$-algebra generated by $M$ and $N$, it can be written as the following form
$$
b=m+\sum_{j=1}^rm_{0,j}x_jm_{1,j},
$$
where $m, m_{0,1}, m_{1,j}\in M$ and each $x_j$ is a reduced word beginning and ending in $\ker\tau_N$.
Hence, by Lemma \ref{L:norm} and noting that $M$ is abelian, one has 
\begin{eqnarray*}
\|A_L(b)-\mathbb{E}_M(b)\| & = & \|(m-\sum_{j=1}^r m_{0, j} A_L(x_j)m_{1,j})-m\| \\
& \leq & \frac{2}{\sqrt{L}}\sum_{j=1}^r\|m_{0,j}\|\|x_j\|\|m_{1,j}\|\rightarrow0,
\end{eqnarray*}
as $L \to\infty$. Since $A_L$ and $\mathbb{E}_M$ are contractions, we have
$$
\|A_L(a)-\mathbb{E}_M(a)\|\rightarrow0, \quad a \in A,
$$
as desired.
\end{proof}

\bibliographystyle{plainurl}

\end{document}